\documentclass{amsart}
\usepackage{amsfonts}

\newtheorem{theorem}{Theorem}
\theoremstyle{plain}

\newtheorem{lemma}{Lemma}

\numberwithin{equation}{section}
\input{tcilatex}

\begin{document}
\title[BVP'S with retarded argument]{On a generalized class of boundary
value problems with delayed argument}
\author{Erdo\u{g}an \c{S}en}
\address{Department of Mathematics, Faculty of Arts and Science, Tekirdag
Namik Kemal University, Tekirda\u{g}, Turkey}
\email{erdogan.math@gmail.com}
\thanks{}

\begin{abstract}
In this work, spectrum and asymptotics of eigenfunctions of a generalized
class of boundary value problems with a delay are obtained.


\vspace{2mm}\noindent \textsc{Keywords and phrases.} Delay differential
equations; transmission conditions; asymptotics of eigenvalues and
eigenfunctions.
\end{abstract}

\maketitle

\section{Formulation of the problem}

In this study we shall investigate discontinuous eigenvalue problems which
consist of Sturm-Liouville equation
\begin{equation}
\left( -p(x)u^{\prime}(x)\right) ^{\prime}+q(x)u(x-\Delta(x))=\lambda
^{2}u(x)=0  \tag{1}
\end{equation}
on $\Omega=\cup\Omega^{\pm}$ with boundary conditions

\begin{equation}
\delta _{10}u(a)-\delta _{11}u^{\prime }(a)-\lambda ^{2}\left( \widetilde{%
\delta }_{10}u(a)-\widetilde{\delta }_{11}u^{\prime }(a)\right) =0,  \tag{2}
\end{equation}%
\begin{equation}
\delta _{20}u(b)-\delta _{21}u^{\prime }(b)+\lambda ^{2}\left( \widetilde{%
\delta }_{20}u(b)-\widetilde{\delta }_{21}u^{\prime }(b)\right) =0  \tag{3}
\end{equation}%
and transmission conditions%
\begin{equation}
\gamma _{10}^{+}u(c+)+\gamma _{10}^{-}u(c-)=0,  \tag{4}
\end{equation}%
\begin{equation*}
\gamma _{20}^{+}u(c+)+\gamma _{21}^{+}u^{\prime }(c+)+\gamma
_{20}^{-}u(c-)+\gamma _{21}^{-}u^{\prime }(c-)
\end{equation*}%
\begin{equation}
-\lambda ^{2}\left( \widetilde{\gamma }_{20}^{+}u(c+)+\widetilde{\gamma }%
_{21}^{+}u^{\prime }(c+)+\widetilde{\gamma }_{20}^{-}u(c-)+\widetilde{\gamma
}_{21}^{-}u^{\prime }(c-)\right) =0,  \tag{5}
\end{equation}%
where $p(x)=p_{1}^{2}$ for $x\in \Omega ^{-}=\left[ a,c\right) $ and $%
p(x)=p_{2}^{2}$ for $x\in \Omega ^{+}=\left( c,b\right] $; the real-valued
function $q(x)$ is continuous in $\Omega ~$and has a finite limit $q(c\pm
)=\lim_{x\rightarrow c\pm }q(x),$ the real valued function $\Delta (x)\geq 0$
continuous in $\Omega $ and has a finite limit $\Delta (c\pm
)=\lim_{x\rightarrow c\pm }\Delta (x)$, $x-\Delta (x)\geq a,$ if $\ x\in
\Omega ^{-};x-\Delta (x)\geq c,$ if $x\in \Omega ^{+};$ $\lambda $ is a real
spectral parameter; $p_{i},$ $\delta _{ij},$ $\widetilde{\delta }_{ij},$ $%
\gamma _{ij}^{\pm },$ $\widetilde{\gamma }_{ij}^{\pm }$ $(i=1,2;$ $j=0,1)$
are arbitrary real numbers such that $\gamma _{10}^{\pm }\left( \gamma
_{21}^{\pm }-\lambda ^{2}\widetilde{\gamma }_{21}^{\pm }\right) \neq 0$.

Sturm-Liouville problems with transmission conditions (also known as
interface conditons, discontinuity conditions, impulse effects) arise in
many applications of mathematical physics. Amongst the applications are
thermal conduction in a thin laminated plate made up of layers of different
materials and diffraction problems [11].

Sturm-Liouville problems with delayed argument is an active area of research
and arise in many realistic models of problems in science, engineering, and
medicine, where there is a time lag or after-effect (see [3]) and find
applications in combustion in a liquid propellant rocket engine [5,10] and
in systems of the type of an electromagnetic circuit-breaker [8,20]. The
articles [1,4,6,9,15,17-19,22] are devoted to investigation of the spectral
properties of eigenvalues and eigenfunctions of the Sturm-Liouville problems
with delayed argument.

The main goal of this paper is to study the spectrum and asymptotics of
eigenfunctions of the problem (1)-(5). Spectral properties of differential
equations with delayed argument which contain such a generalized boundary
and transmission conditions have not been studied yet. So, the results
obtained in this work are extension and generalization of previous works in
the literature. For example, if we take $\Delta (x)\equiv 0$ and/or $%
\widetilde{\delta }_{ij}=0$ $\left( i=1,2;\text{ }j=0,1\right) $ and/or $%
\widetilde{\gamma }_{2j}^{\pm }=0$ $\left( \text{ }j=0,1\right) $ and/or $%
p(x)\equiv 1$ then the asymptotic formulas for eigenvalues and
eigenfunctions correspond to those for the classical Sturm-Liouville problem
[2,7,12,13,14,21]. Moreover, results and methods of these kind of problems can
be useful for investigating the inverse problems for partial differential
equations.

Let $\vartheta ^{-}(x,\lambda )$ be a solution of Eq. (1) on $\overline{%
\Omega ^{-}}=\Omega ^{-}\cup \left\{ c\right\} ,$ satisfying the initial
conditions%
\begin{equation}
\vartheta ^{-}\left( a,\lambda \right) =\delta _{11}-\lambda ^{2}\widetilde{%
\delta }_{11},\text{ }\frac{\partial \vartheta ^{-}\left( a,\lambda \right)
}{\partial x}=\delta _{10}-\lambda ^{2}\widetilde{\delta }_{10}.  \tag{6}
\end{equation}%
The conditions (6) define a unique solution of Eq. (1) on $\overline{\Omega
^{-}}$ [16].

After defining the above solution we shall define the solution $\vartheta
^{+}\left( x,\lambda \right) $ of Eq. (1) on $\overline{\Omega ^{+}}=\Omega
^{+}\cup \left\{ c\right\} $ by means of the solution $\vartheta ^{-}\left(
x,\lambda \right) $ using the initial conditions%
\begin{equation}
\vartheta ^{+}\left( c+,\lambda \right) =-\frac{\gamma _{10}^{-}}{\gamma
_{10}^{+}}\vartheta ^{-}\left( c-,\lambda \right) ,  \tag{7}
\end{equation}%
\begin{equation}
\frac{\partial \vartheta ^{+}\left( c+,\lambda \right) }{\partial x}=\frac{1%
}{\gamma _{10}^{+}\left( \gamma _{21}^{+}-\lambda ^{2}\widetilde{\gamma }%
_{21}^{+}\right) }  \notag
\end{equation}%
\begin{equation}
\times \left[ \gamma _{10}^{+}\left( \lambda ^{2}\widetilde{\gamma }%
_{21}^{-}-\gamma _{21}^{-}\right) \frac{\partial \vartheta ^{-}\left(
c-,\lambda \right) }{\partial x}+\left( \gamma _{10}^{+}\left( \lambda ^{2}%
\widetilde{\gamma }_{20}^{-}-\gamma _{20}^{-}\right) -\gamma _{10}^{-}\left(
\lambda ^{2}\widetilde{\gamma }_{20}^{+}-\gamma _{20}^{+}\right) \right)
\vartheta ^{-}\left( c-,\lambda \right) \right]  \tag{8}
\end{equation}%
The conditions (7)-(8) are defined as a unique solution of Eq. (1) on $%
\overline{\Omega ^{+}}.$

Consequently, the function $\vartheta \left( x,\lambda \right) $ is defined
on $\Omega $ by the equality%
\begin{equation*}
\vartheta (x,\lambda )=\left\{
\begin{array}{ll}
\vartheta ^{-}(x,\lambda ), & x\in \Omega ^{-}, \\
\vartheta ^{+}(x,\lambda ), & x\in \Omega ^{+}%
\end{array}%
\right.
\end{equation*}%
is a solution of the Eq. (1) on $\Omega ;$ which satisfies one of the
boundary conditions and both transmission conditions.

\section{Spectrum and Asymptotics of Eigenfunctions}

We begin by writing the problem (1)-(5) in terms of the following equivalent
integral equations.

\begin{lemma}
Let $\vartheta\left( x,\lambda\right) $ be a solution of Eq.$(1)$ and\ $%
\lambda>0.$ Then the following integral equations hold:\
\begin{align}
\vartheta^{-}(x,\lambda) & =\left( \delta_{11}-\lambda^{2}\widetilde {\delta}%
_{11}\right) \cos\frac{\lambda\left( x-a\right) }{p_{1}}+\frac{p_{1}\left(
\delta_{10}-\lambda^{2}\widetilde{\delta}_{10}\right) }{\lambda}\sin\frac{%
\lambda\left( x-a\right) }{p_{1}}  \notag \\
& +\frac{1}{p_{1}\lambda}\int\limits_{a}^{{x}}q\left( \tau\right) \sin \frac{%
\lambda\left( x-\tau\right) }{p_{1}}\vartheta^{-}\left( \tau -\Delta\left(
\tau\right) ,\lambda\right) d\tau,  \tag{9}
\end{align}%
\begin{equation}
\vartheta^{+}(x,\lambda)=-\frac{\gamma_{10}^{-}\vartheta^{-}(c-,\lambda )}{%
\gamma_{10}^{+}}\cos\frac{\lambda\left( x-c\right) }{p_{2}}+\frac{p_{2}}{%
\lambda\gamma_{10}^{+}\left( \gamma_{21}^{+}-\lambda^{2}\widetilde{\gamma }%
_{21}^{+}\right) }  \notag
\end{equation}%
\begin{equation*}
\times\left[ \gamma_{10}^{+}\left( \lambda^{2}\widetilde{\gamma}%
_{21}^{-}-\gamma_{21}^{-}\right) \frac{\partial\vartheta^{-}(c-,\lambda)}{%
\partial x}+\left( \gamma_{10}^{+}\left( \lambda^{2}\widetilde{\gamma}%
_{20}^{-}-\gamma_{20}^{-}\right) -\gamma_{10}^{-}\left( \lambda^{2}%
\widetilde{\gamma}_{20}^{+}-\gamma_{20}^{+}\right) \right) \vartheta
^{-}\left( c-,\lambda\right) \right]
\end{equation*}%
\begin{equation}
\times\sin\frac{\lambda\left( x-c\right) }{p_{2}}+\frac{1}{p_{2}\lambda}%
\int\limits_{c+}^{{x}}q\left( \tau\right) \sin\frac{\lambda\left(
x-\tau\right) }{p_{2}}\vartheta^{+}\left( \tau-\Delta\left( \tau\right)
,\lambda\right) d\tau  \tag{10}
\end{equation}
\end{lemma}

\begin{proof}
To prove this, it is enough to substitute $\>\lambda^{2}\vartheta^{\pm}(%
\tau,\lambda)+\frac{\partial^{2}\vartheta^{\pm}\left( \tau,\lambda\right) }{%
\partial\tau^{2}}\>$instead of $q(\tau)\vartheta^{\pm}(\tau-\Delta
(\tau),\lambda)$ in (9) and (10) respectively and integrate by parts twice.
\end{proof}

From Lemma 1, using the well-known successive approximation method, it is
easy to obtain the following asymptotic expressions of fundamental solutions.

\begin{lemma}
The following asymptotic estimates
\begin{equation*}
\vartheta ^{-}\left( x,\lambda \right) =-\lambda ^{2}\widetilde{\delta }%
_{11}\cos \frac{\lambda \left( x-a\right) }{p_{1}}+O\left( \lambda \right) ,
\end{equation*}%
\begin{equation*}
\frac{\partial \vartheta ^{-}\left( x,\lambda \right) }{\partial x}=\frac{%
\lambda ^{3}\widetilde{\delta }_{11}}{p_{1}}\sin \frac{\lambda \left(
x-a\right) }{p_{1}}+O\left( \lambda ^{2}\right) ,
\end{equation*}%
\begin{equation*}
\vartheta ^{+}\left( x,\lambda \right) =\frac{\lambda ^{4}p_{2}\widetilde{%
\delta }_{11}\widetilde{\gamma }_{21}^{-}}{p_{1}\gamma _{21}^{+}}\sin \frac{%
\lambda \left( c-a\right) }{p_{1}}\sin \frac{\lambda \left( x-c\right) }{%
p_{2}}+O\left( \lambda ^{3}\right) ,
\end{equation*}%
\begin{equation*}
\frac{\partial \vartheta ^{+}\left( x,\lambda \right) }{\partial x}=\frac{%
\lambda ^{5}\widetilde{\delta }_{11}\widetilde{\gamma }_{21}^{-}}{%
p_{1}\gamma _{21}^{+}}\sin \frac{\lambda \left( c-a\right) }{p_{1}}\cos
\frac{\lambda \left( x-c\right) }{p_{2}}+O\left( \lambda ^{4}\right)
\end{equation*}%
are valid as $\lambda \rightarrow \infty $.
\end{lemma}

The function $\vartheta (x,\lambda )\>$defined in introduction is a
nontrivial solution of Eq. (1) satisfying conditions (2), (4) and (5).
Putting$\>\vartheta (x,\lambda )\>$into (3), we get the characteristic
equation%
\begin{equation}
\Xi (\lambda )\equiv \delta _{20}\vartheta ^{+}(b,\lambda )-\delta _{21}%
\frac{\partial \vartheta ^{+}\left( b,\lambda \right) }{\partial x}+\lambda
^{2}\left( \widetilde{\delta }_{20}\vartheta ^{+}(b,\lambda )-\widetilde{%
\delta }_{21}\frac{\partial \vartheta ^{+}\left( b,\lambda \right) }{%
\partial x}\right) =0.  \tag{11}
\end{equation}

Thus the set of eigenvalues of boundary-value problem (1)-(5) coincides with
the set of real roots of Eq. (11).

\begin{theorem}
The problem $(1)-(5)$ has an infinite set of positive eigenvalues.
\end{theorem}

\begin{proof}
Differentiating (9) with respect to$\>x$, we get%
\begin{equation*}
\frac{\partial \vartheta ^{-}(x,\lambda )}{\partial x}=\frac{\lambda \left(
\delta _{11}-\lambda ^{2}\widetilde{\delta }_{11}\right) }{p_{1}}\sin \frac{%
\lambda \left( x-a\right) }{p_{1}}+\left( \delta _{10}-\lambda ^{2}%
\widetilde{\delta }_{10}\right) \cos \frac{\lambda \left( x-a\right) }{p_{1}}
\end{equation*}%
\begin{equation}
+\frac{1}{p_{1}^{2}}\int\limits_{a}^{{x}}q\left( \tau \right) \cos \frac{%
\lambda \left( x-\tau \right) }{p_{1}}\vartheta ^{-}\left( \tau -\Delta
\left( \tau \right) ,\lambda \right) d\tau .  \tag{12}
\end{equation}%
Differentiating (10) with respect to$\>x$, we get%
\begin{equation*}
\frac{\partial \vartheta ^{+}(x,\lambda )}{\partial x}=\frac{\lambda \gamma
_{10}^{-}\vartheta ^{-}(c-,\lambda )}{p_{2}\gamma _{10}^{+}}\sin \frac{%
\lambda \left( x-c\right) }{p_{2}}+\frac{1}{\gamma _{10}^{+}\left( \gamma
_{21}^{+}-\lambda ^{2}\widetilde{\gamma }_{21}^{+}\right) }
\end{equation*}%
\begin{equation*}
\times \left[ \gamma _{10}^{+}\left( \lambda ^{2}\widetilde{\gamma }%
_{21}^{-}-\gamma _{21}^{-}\right) \frac{\partial \vartheta ^{-}(c-,\lambda )%
}{\partial x}+\left( \gamma _{10}^{+}\left( \lambda ^{2}\widetilde{\gamma }%
_{20}^{-}-\gamma _{20}^{-}\right) -\gamma _{10}^{-}\left( \lambda ^{2}%
\widetilde{\gamma }_{20}^{+}-\gamma _{20}^{+}\right) \right) \vartheta
^{-}\left( c-,\lambda \right) \right]
\end{equation*}%
\begin{equation}
\times \cos \frac{\lambda \left( x-c\right) }{p_{2}}+\frac{1}{p_{2}^{2}}%
\int\limits_{c+}^{{x}}q\left( \tau \right) \cos \frac{\lambda \left( x-\tau
\right) }{p_{2}}\vartheta ^{+}\left( \tau -\Delta \left( \tau \right)
,\lambda \right) d\tau .  \tag{13}
\end{equation}%
Putting the expressions (9), (10), (12) and (13) into (11), we get%
\begin{equation*}
\Xi (\lambda )\equiv \delta _{20}\left[ -\frac{\gamma _{10}^{-}}{\gamma
_{10}^{+}}\left( \left( \delta _{11}-\lambda ^{2}\widetilde{\delta }%
_{11}\right) \cos \frac{\lambda \left( c-a\right) }{p_{1}}+\frac{p_{1}\left(
\delta _{10}-\lambda ^{2}\widetilde{\delta }_{10}\right) }{\lambda }\sin
\frac{\lambda \left( c-a\right) }{p_{1}}\right. \right.
\end{equation*}%
\begin{equation*}
\left. +\frac{1}{p_{1}\lambda }\int\limits_{a}^{{c}}q\left( \tau \right)
\sin \frac{\lambda \left( c-\tau \right) }{p_{1}}\vartheta ^{-}\left( \tau
-\Delta \left( \tau \right) ,\lambda \right) d\tau \right) \cos \frac{%
\lambda \left( b-c\right) }{p_{2}}+\frac{p_{2}}{\lambda \gamma
_{10}^{+}\left( \gamma _{21}^{+}-\lambda ^{2}\widetilde{\gamma }%
_{21}^{+}\right) }
\end{equation*}%
\begin{equation}
\times \left( \gamma _{10}^{+}\left( \lambda ^{2}\widetilde{\gamma }%
_{21}^{-}-\gamma _{21}^{-}\right) \left( \frac{\lambda \left( \delta
_{11}-\lambda ^{2}\widetilde{\delta }_{11}\right) }{p_{1}}\sin \frac{\lambda
\left( c-a\right) }{p_{1}}+\left( \delta _{10}-\lambda ^{2}\widetilde{\delta
}_{10}\right) \cos \frac{\lambda \left( c-a\right) }{p_{1}}\right. \right.
\notag
\end{equation}%
\begin{equation*}
\left. +\frac{1}{p_{1}^{2}}\int\limits_{a}^{{c}}q\left( \tau \right) \cos
\frac{\lambda \left( c-\tau \right) }{p_{1}}\vartheta ^{-}\left( \tau
-\Delta \left( \tau \right) ,\lambda \right) d\tau \right) +\left( \gamma
_{10}^{+}\left( \lambda ^{2}\widetilde{\gamma }_{20}^{-}-\gamma
_{20}^{-}\right) -\gamma _{10}^{-}\left( \lambda ^{2}\widetilde{\gamma }%
_{20}^{+}-\gamma _{20}^{+}\right) \right)
\end{equation*}%
\begin{equation*}
\times \left( -\frac{\gamma _{10}^{-}}{\gamma _{10}^{+}}\left( \left( \delta
_{11}-\lambda ^{2}\widetilde{\delta }_{11}\right) \cos \frac{\lambda \left(
c-a\right) }{p_{1}}+\frac{p_{1}\left( \delta _{10}-\lambda ^{2}\widetilde{%
\delta }_{10}\right) }{\lambda }\sin \frac{\lambda \left( c-a\right) }{p_{1}}%
\right. \right.
\end{equation*}%
\begin{equation*}
\left. \left. +\frac{1}{p_{1}\lambda }\int\limits_{a}^{{c}}q\left( \tau
\right) \sin \frac{\lambda \left( c-\tau \right) }{p_{1}}\vartheta
^{-}\left( \tau -\Delta \left( \tau \right) ,\lambda \right) d\tau \right)
\right) \sin \frac{\lambda \left( b-c\right) }{p_{2}}
\end{equation*}%
\begin{equation*}
\left. +\frac{1}{p_{2}\lambda }\int\limits_{c+}^{{b}}q\left( \tau \right)
\sin \frac{\lambda \left( b-\tau \right) }{p_{2}}\vartheta ^{+}\left( \tau
-\Delta \left( \tau \right) ,\lambda \right) d\tau \right]
\end{equation*}%
\begin{equation*}
-\delta _{21}\left[ \frac{\lambda \gamma _{10}^{-}}{p_{2}\gamma _{10}^{+}}%
\left( \left( \left( \delta _{11}-\lambda ^{2}\widetilde{\delta }%
_{11}\right) \cos \frac{\lambda \left( c-a\right) }{p_{1}}+\frac{p_{1}\left(
\delta _{10}-\lambda ^{2}\widetilde{\delta }_{10}\right) }{\lambda }\sin
\frac{\lambda \left( c-a\right) }{p_{1}}\right. \right. \right.
\end{equation*}%
\begin{equation*}
\left. \left. +\frac{1}{p_{1}\lambda }\int\limits_{a}^{{c}}q\left( \tau
\right) \sin \frac{\lambda \left( c-\tau \right) }{p_{1}}\vartheta
^{-}\left( \tau -\Delta \left( \tau \right) ,\lambda \right) d\tau \right)
\right) \sin \frac{\lambda \left( b-c\right) }{p_{2}}+\frac{1}{\gamma
_{10}^{+}\left( \gamma _{21}^{+}-\lambda ^{2}\widetilde{\gamma }%
_{21}^{+}\right) }
\end{equation*}%
\begin{equation*}
\times \left( \gamma _{10}^{+}\left( \lambda ^{2}\widetilde{\gamma }%
_{21}^{-}-\gamma _{21}^{-}\right) \left( \frac{\lambda \left( \delta
_{11}-\lambda ^{2}\widetilde{\delta }_{11}\right) }{p_{1}}\sin \frac{\lambda
\left( c-a\right) }{p_{1}}+\left( \delta _{10}-\lambda ^{2}\widetilde{\delta
}_{10}\right) \cos \frac{\lambda \left( c-a\right) }{p_{1}}\right. \right.
\end{equation*}%
\begin{equation*}
\left. +\frac{1}{p_{1}^{2}}\int\limits_{a}^{{c}}q\left( \tau \right) \cos
\frac{\lambda \left( c-\tau \right) }{p_{1}}\vartheta ^{-}\left( \tau
-\Delta \left( \tau \right) ,\lambda \right) d\tau \right) +\left( \gamma
_{10}^{+}\left( \lambda ^{2}\widetilde{\gamma }_{20}^{-}-\gamma
_{20}^{-}\right) -\gamma _{10}^{-}\left( \lambda ^{2}\widetilde{\gamma }%
_{20}^{+}-\gamma _{20}^{+}\right) \right)
\end{equation*}%
\begin{equation*}
\times \left( \left( \delta _{11}-\lambda ^{2}\widetilde{\delta }%
_{11}\right) \cos \frac{\lambda \left( c-a\right) }{p_{1}}+\frac{p_{1}\left(
\delta _{10}-\lambda ^{2}\widetilde{\delta }_{10}\right) }{\lambda }\sin
\frac{\lambda \left( c-a\right) }{p_{1}}\right.
\end{equation*}%
\begin{equation*}
\left. \left. +\frac{1}{p_{1}\lambda }\int\limits_{a}^{{c}}q\left( \tau
\right) \sin \frac{\lambda \left( c-\tau \right) }{p_{1}}\vartheta
^{-}\left( \tau -\Delta \left( \tau \right) ,\lambda \right) d\tau \right)
\right) \cos \frac{\lambda \left( b-c\right) }{p_{2}}
\end{equation*}%
\begin{equation*}
\left. +\frac{1}{p_{2}^{2}}\int\limits_{c+}^{{b}}q\left( \tau \right) \cos
\frac{\lambda \left( b-\tau \right) }{p_{2}}\vartheta ^{+}\left( \tau
-\Delta \left( \tau \right) ,\lambda \right) d\tau \right]
\end{equation*}%
\begin{equation*}
+\lambda ^{2}\left\{ \widetilde{\delta }_{20}\left[ -\frac{\gamma _{10}^{-}}{%
\gamma _{10}^{+}}\left( \left( \delta _{11}-\lambda ^{2}\widetilde{\delta }%
_{11}\right) \cos \frac{\lambda \left( c-a\right) }{p_{1}}+\frac{p_{1}\left(
\delta _{10}-\lambda ^{2}\widetilde{\delta }_{10}\right) }{\lambda }\sin
\frac{\lambda \left( c-a\right) }{p_{1}}\right. \right. \right.
\end{equation*}%
\begin{equation*}
\left. +\frac{1}{p_{1}\lambda }\int\limits_{a}^{{c}}q\left( \tau \right)
\sin \frac{\lambda \left( c-\tau \right) }{p_{1}}\vartheta ^{-}\left( \tau
-\Delta \left( \tau \right) ,\lambda \right) d\tau \right) \cos \frac{%
\lambda \left( b-c\right) }{p_{2}}+\frac{p_{2}}{\lambda \gamma
_{10}^{+}\left( \gamma _{21}^{+}-\lambda ^{2}\widetilde{\gamma }%
_{21}^{+}\right) }
\end{equation*}%
\begin{equation}
\times \left( \gamma _{10}^{+}\left( \lambda ^{2}\widetilde{\gamma }%
_{21}^{-}-\gamma _{21}^{-}\right) \left( \frac{\lambda \left( \delta
_{11}-\lambda ^{2}\widetilde{\delta }_{11}\right) }{p_{1}}\sin \frac{\lambda
\left( c-a\right) }{p_{1}}+\left( \delta _{10}-\lambda ^{2}\widetilde{\delta
}_{10}\right) \cos \frac{\lambda \left( c-a\right) }{p_{1}}\right. \right.
\notag
\end{equation}%
\begin{equation*}
\left. +\frac{1}{p_{1}^{2}}\int\limits_{a}^{{c}}q\left( \tau \right) \cos
\frac{\lambda \left( c-\tau \right) }{p_{1}}\vartheta ^{-}\left( \tau
-\Delta \left( \tau \right) ,\lambda \right) d\tau \right) +\left( \gamma
_{10}^{+}\left( \lambda ^{2}\widetilde{\gamma }_{20}^{-}-\gamma
_{20}^{-}\right) -\gamma _{10}^{-}\left( \lambda ^{2}\widetilde{\gamma }%
_{20}^{+}-\gamma _{20}^{+}\right) \right)
\end{equation*}%
\begin{equation*}
\times \left( -\frac{\gamma _{10}^{-}}{\gamma _{10}^{+}}\left( \left( \delta
_{11}-\lambda ^{2}\widetilde{\delta }_{11}\right) \cos \frac{\lambda \left(
c-a\right) }{p_{1}}+\frac{p_{1}\left( \delta _{10}-\lambda ^{2}\widetilde{%
\delta }_{10}\right) }{\lambda }\sin \frac{\lambda \left( c-a\right) }{p_{1}}%
\right. \right.
\end{equation*}%
\begin{equation*}
\left. \left. +\frac{1}{p_{1}\lambda }\int\limits_{a}^{{c}}q\left( \tau
\right) \sin \frac{\lambda \left( c-\tau \right) }{p_{1}}\vartheta
^{-}\left( \tau -\Delta \left( \tau \right) ,\lambda \right) d\tau \right)
\right) \sin \frac{\lambda \left( b-c\right) }{p_{2}}
\end{equation*}%
\begin{equation*}
\left. +\frac{1}{p_{2}\lambda }\int\limits_{c+}^{{b}}q\left( \tau \right)
\sin \frac{\lambda \left( b-\tau \right) }{p_{2}}\vartheta ^{+}\left( \tau
-\Delta \left( \tau \right) ,\lambda \right) d\tau \right]
\end{equation*}%
\begin{equation*}
-\widetilde{\delta }_{21}\left[ \frac{\lambda \gamma _{10}^{-}}{p_{2}\gamma
_{10}^{+}}\left( \left( \left( \delta _{11}-\lambda ^{2}\widetilde{\delta }%
_{11}\right) \cos \frac{\lambda \left( c-a\right) }{p_{1}}+\frac{p_{1}\left(
\delta _{10}-\lambda ^{2}\widetilde{\delta }_{10}\right) }{\lambda }\sin
\frac{\lambda \left( c-a\right) }{p_{1}}\right. \right. \right.
\end{equation*}%
\begin{equation*}
\left. \left. +\frac{1}{p_{1}\lambda }\int\limits_{a}^{{c}}q\left( \tau
\right) \sin \frac{\lambda \left( c-\tau \right) }{p_{1}}\vartheta
^{-}\left( \tau -\Delta \left( \tau \right) ,\lambda \right) d\tau \right)
\right) \sin \frac{\lambda \left( b-c\right) }{p_{2}}+\frac{1}{\gamma
_{10}^{+}\left( \gamma _{21}^{+}-\lambda ^{2}\widetilde{\gamma }%
_{21}^{+}\right) }
\end{equation*}%
\begin{equation*}
\times \left( \gamma _{10}^{+}\left( \lambda ^{2}\widetilde{\gamma }%
_{21}^{-}-\gamma _{21}^{-}\right) \left( \frac{\lambda \left( \delta
_{11}-\lambda ^{2}\widetilde{\delta }_{11}\right) }{p_{1}}\sin \frac{\lambda
\left( c-a\right) }{p_{1}}+\left( \delta _{10}-\lambda ^{2}\widetilde{\delta
}_{10}\right) \cos \frac{\lambda \left( c-a\right) }{p_{1}}\right. \right.
\end{equation*}%
\begin{equation*}
\left. +\frac{1}{p_{1}^{2}}\int\limits_{a}^{{c}}q\left( \tau \right) \cos
\frac{\lambda \left( c-\tau \right) }{p_{1}}\vartheta ^{-}\left( \tau
-\Delta \left( \tau \right) ,\lambda \right) d\tau \right) +\left( \gamma
_{10}^{+}\left( \lambda ^{2}\widetilde{\gamma }_{20}^{-}-\gamma
_{20}^{-}\right) -\gamma _{10}^{-}\left( \lambda ^{2}\widetilde{\gamma }%
_{20}^{+}-\gamma _{20}^{+}\right) \right)
\end{equation*}%
\begin{equation*}
\times \left( \left( \delta _{11}-\lambda ^{2}\widetilde{\delta }%
_{11}\right) \cos \frac{\lambda \left( c-a\right) }{p_{1}}+\frac{p_{1}\left(
\delta _{10}-\lambda ^{2}\widetilde{\delta }_{10}\right) }{\lambda }\sin
\frac{\lambda \left( c-a\right) }{p_{1}}\right.
\end{equation*}%
\begin{equation*}
\left. \left. +\frac{1}{p_{1}\lambda }\int\limits_{a}^{{c}}q\left( \tau
\right) \sin \frac{\lambda \left( c-\tau \right) }{p_{1}}\vartheta
^{-}\left( \tau -\Delta \left( \tau \right) ,\lambda \right) d\tau \right)
\right) \cos \frac{\lambda \left( b-c\right) }{p_{2}}
\end{equation*}%
\begin{equation*}
\left. \left. +\frac{1}{p_{2}^{2}}\int\limits_{c+}^{{b}}q\left( \tau \right)
\cos \frac{\lambda \left( b-\tau \right) }{p_{2}}\vartheta ^{+}\left( \tau
-\Delta \left( \tau \right) ,\lambda \right) d\tau \right] \right\} =0.
\end{equation*}%
From Lemma 2, the following equalities
\begin{equation*}
\vartheta ^{-}\left( \tau -\Delta \left( \tau \right) ,\lambda \right)
=-\lambda ^{2}\widetilde{\delta }_{11}\cos \frac{\lambda \left( \tau -\Delta
\left( \tau \right) -a\right) }{p_{1}}+O\left( \lambda \right) ,
\end{equation*}%
\begin{equation*}
\frac{\partial \vartheta ^{-}\left( \tau -\Delta \left( \tau \right)
,\lambda \right) }{\partial x}=\frac{\lambda ^{3}\widetilde{\delta }_{11}}{%
p_{1}}\sin \frac{\lambda \left( \tau -\Delta \left( \tau \right) -a\right) }{%
p_{1}}+O\left( \lambda ^{2}\right) ,
\end{equation*}%
\begin{equation*}
\vartheta ^{+}\left( \tau -\Delta \left( \tau \right) ,\lambda \right) =%
\frac{\lambda ^{4}p_{2}\widetilde{\delta }_{11}\widetilde{\gamma }_{21}^{-}}{%
p_{1}\gamma _{21}^{+}}\sin \frac{\lambda \left( c-a\right) }{p_{1}}\sin
\frac{\lambda \left( \tau -\Delta \left( \tau \right) -c\right) }{p_{2}}%
+O\left( \lambda ^{3}\right) ,
\end{equation*}%
\begin{equation*}
\frac{\partial \vartheta ^{+}\left( \tau -\Delta \left( \tau \right)
,\lambda \right) }{\partial x}=\frac{\lambda ^{5}\widetilde{\delta }_{11}%
\widetilde{\gamma }_{21}^{-}}{p_{1}\gamma _{21}^{+}}\sin \frac{\lambda
\left( c-a\right) }{p_{1}}\cos \frac{\lambda \left( \tau -\Delta \left( \tau
\right) -c\right) }{p_{2}}+O\left( \lambda ^{4}\right)
\end{equation*}%
hold as $\lambda \rightarrow \infty $ and using these approximations, we have%
\begin{equation}
\Xi (\lambda )=-\frac{\widetilde{\delta }_{11}\widetilde{\delta }_{21}%
\widetilde{\gamma }_{21}^{-}\lambda ^{7}}{\gamma _{21}^{+}p_{1}}\sin \frac{%
\lambda \left( c-a\right) }{p_{1}}\cos \frac{\lambda \left( b-c\right) }{%
p_{2}}+O\left( \lambda ^{6}\right) =0.  \tag{14}
\end{equation}

Let $\>\lambda \>$ be sufficiently large. Obviously, for large$\>\lambda \>$%
Eq. (14) has, evidently, an infinite set of roots. The proof is complete.
\end{proof}

By Theorem 2 we conclude that the problem (1)-(5) has infinitely many
nontrivial solutions.

Solving the Eq. (14), we have

\begin{equation*}
\Gamma =\left\{ \lambda _{n}:\lambda _{n}=\frac{p_{2}\pi \left( n+\frac{1}{2}%
\right) }{b-c}+O\left( \frac{1}{n}\right) \text{ or }\lambda _{n}=\frac{%
p_{1}\pi n}{c-a}+O\left( \frac{1}{n}\right) ,\text{ }n=1,2,...\right\}
\end{equation*}%
for the spectrum of (1)-(5).

Now we are ready to present asymptotic expressions of eigenfunctions. Using
Lemma 2 and replacing $\lambda $ by $\lambda _{n}\in \Gamma $ we obtain the
next theorem. We see that there correspond two eigenfunctions for each $n.$

\begin{theorem}
The following asymptotic formulas hold for eigenfunctions of
boundary-value-transmission problem (1)-(5) for each $x\in \Omega $ and $%
\lambda _{n}\in \Gamma $ $\left( n=1,2,...\right) $:%
\begin{align*}
\vartheta _{\left( 1\right) }^{-}\left( x,\lambda _{n}\right) & =-\frac{%
n^{2}\pi ^{2}p_{1}^{2}\widetilde{\delta }_{11}}{\left( c-a\right) ^{2}}\cos
\frac{n\pi \left( x-a\right) }{c-a}+O(n), \\
\vartheta _{\left( 2\right) }^{-}\left( x,\lambda _{n}\right) & =-\frac{%
\left( n+\frac{1}{2}\right) ^{2}\pi ^{2}p_{2}^{2}\widetilde{\delta }_{11}}{%
\left( b-c\right) ^{2}}\cos \frac{\left( n+\frac{1}{2}\right) \pi
p_{2}\left( x-a\right) }{p_{1}\left( b-c\right) }+O(n),
\end{align*}%
\begin{align*}
\vartheta _{(1)}^{+}\left( x,\lambda _{n}\right) & =\frac{n^{3}\pi
^{4}p_{1}^{2}p_{2}\widetilde{\delta }_{11}\widetilde{\gamma }_{21}^{-}}{%
\gamma _{21}^{+}\left( c-a\right) ^{3}}\sin \frac{n\pi p_{1}\left(
x-c\right) }{p_{2}\left( c-a\right) }+O\left( n^{2}\right) , \\
\vartheta _{(2)}^{+}\left( x,\lambda _{n}\right) & =\frac{\left( n+\frac{1}{2%
}\right) ^{4}\pi ^{4}p_{2}^{5}\widetilde{\delta }_{11}\widetilde{\gamma }%
_{21}^{-}}{p_{1}\gamma _{21}^{+}\left( b-c\right) ^{4}}\left\{ \sin \frac{%
\left( n+\frac{1}{2}\right) \pi p_{2}\left( c-a\right) }{p_{1}\left(
b-c\right) }\right. \\
& \times \left[ \cos \frac{\left( n+\frac{1}{2}\right) \pi \left( x-c\right)
}{b-c}-\frac{x-c}{np_{2}}\sin \frac{\left( n+\frac{1}{2}\right) \pi \left(
x-c\right) }{b-c}\right] \\
& -\left. \frac{c-a}{np_{1}}\cos \frac{\left( n+\frac{1}{2}\right) \pi
p_{2}\left( c-a\right) }{p_{1}\left( b-c\right) }\cos \frac{\left( n+\frac{1%
}{2}\right) \pi \left( x-c\right) }{b-c}\right\} +O\left( n^{2}\right) .
\end{align*}
\end{theorem}

\end{document}